\documentclass[reqno, 10.5pt]{amsart}
\usepackage{amsmath}
\usepackage{amscd}
\usepackage{amssymb}
\numberwithin{equation}{section}

\newcommand{\na}{\nabla}

\newcommand{\RR}{\mathbb{R}}

\newtheorem{theorem}{Theorem}[section]
\newtheorem{theorem/definition}{Theorem/Definition}[section]
\newtheorem{proposition}{Proposition}[section]
\newtheorem{lemma}{Lemma}[section]

\theoremstyle{remark}
\newtheorem{remark}{Remark}[section]
\theoremstyle{definition}

\begin{document}
\title[Curvature estimates for 4D gradient expanding Ricci solitons]
{Curvature Estimates for four-dimensional complete gradient expanding Ricci solitons}
\author{HUAI-DONG CAO AND Tianbo Liu}
\address{Department of Mathematics,  Lehigh University,
Bethlehem, PA 18015, USA}
\email{huc2@lehigh.edu; til615@lehigh.edu}

\thanks{Research of Huai-Dong Cao was partially supported by a Simons Foundation Collaboration Grant (\#586694 HC)}

\begin{abstract}
In this paper, we derive curvature estimates for $4$-dimensional complete gradient expanding Ricci solitons with nonnegative Ricci curvature (outside a compact set $K$). More precisely, we prove that the norm of the curvature tensor $Rm$ and its covariant derivative $\na Rm$ can be bounded by the scalar curvature $R$ by $|Rm|\le C_aR^{a}$ and  $|\na Rm|\le C_aR^{a}$ (on $M\setminus K$), for any $0\leq a<1$ and some constant $C_a>0$.  Moreover, if the scalar curvature has at most polynomial decay at infinity, then $|Rm|\le CR$  (on $M\setminus K$). As an application, it follows that if a $4$-dimensional complete gradient expanding Ricci soliton $(M^4,g,f)$ has nonnegative Ricci curvature and finite asymptotic scalar curvature ratio then it has finite asymptotic curvature ratio, hence admits $C^{1, \alpha}$ asymptotic cones at infinity $(0<\alpha <1)$ according to Chen-Deruelle \cite{ChenDer}.

\end{abstract}

\maketitle
\date{}

\section{Introduction}

A complete Riemannian manifold $(M^n, g)$ is called a {\it gradient Ricci soliton} if there exists
a smooth function $f$ on $M^n$ such that the Ricci tensor $Rc$
of the metric $g$ satisfies the equation
$$ Rc + \nabla^2 f=\lambda g$$
for some $\lambda \in{\mathbb R}$, where $\na^2 f$ denotes the Hessian of $f$. The Ricci soliton is said to be expanding, or steady, or shrinking if $\lambda<0$, or  $\lambda=0$, or $\lambda>0$. 
The function $f$ is called a {\it potential function} of the gradient Ricci soliton.
Clearly, when $f$ is a constant the gradient Ricci soliton $(M^n, g, f)$ is simply an Einstein manifold. 

Gradient Ricci solitons are a natural extension of Einstein metrics. They are also self-similar solutions to Hamilton's  Ricci flow and play an important role in the study of formation of singularities \cite{Ha95F}.  In particular, gradient expanding Ricci solitons may arise as Type III singularity models in the Ricci flow (under suitable positive curvature assumptions)  \cite{Cao97, ChenZhu}, and over which the matrix Li-Yau-Hamilton inequality (also known as the matrix  differential Harnack inequality) becomes equality \cite{Ha93, Cao92}.  Indeed, Schulze-Simon \cite {SS13}  obtained gradient expanding solitons out of the asymptotic cones at infinity of solutions to the Ricci flow on complete noncompact Riemannian manifolds with bounded and non-negative curvature operator and positive asymptotic volume ratio. See also more recent work \cite{Der16, CD20}.

There have been a lot of advances in investigating the geometry of gradient shrinking and steady Ricci solitons in the past two decades, while the research activities  on expanding solitons have also picked up in recent years.  In this paper, we are mainly concerned with $4$-dimensional complete noncompact gradient expanding solitons with nonnegative Ricci curvature.  
By scaling, we shall assume throughout the paper that $\lambda=-\frac1 2$, so the gradient expanding Ricci soliton equation is of the form
\begin{equation}
Rc+\nabla^2 f= -\frac 1 2 g
\end{equation}

Like steady solitons,  compact expanding solitons are necessarily Einstein.  Hence the study of expanding solitons has been focused on complete noncompact ones.  
The simplest example of a complete (noncompact) gradient expanding  soliton is the {\it Gaussian expander }on $\RR^n$, with the standard flat metric and potential function $f(x)=-|x|^2/4$. In addition, Bryant \cite{Bryant} and the first author \cite{Cao97} have constructed non-flat rotationally symmetric expanding gradient Ricci and K\"ahler-Ricci solitons on $\RR^n$ and $\mathbb {C}^n$, respectively. In particular, they constructed one-parameter families of expanding solitons with positive sectional curvature that are asymptotic to a cone at infinity. Moreover,  the constructions in \cite{Cao94, Cao97} have been extended by Feldman-Ilmanen-Knopf \cite{FIK} to a construction of gradient  expanding K\"ahler-Ricci solitons on the complex line bundles $O(-k)$ ($k>n$) over the complex projective space ${\mathbb C}P^n$ ($n\geq 1$), and further generalized by Dancer-Wang  \cite{DW11}. 
For additional examples and other constructions, see, e.g., \cite{PTV00, Lauret01, GK04, BL07, DW09, FW11, BDGW15, AK19, Win19b} and the references therein. Note especially that the examples constructed in \cite{Lauret01} and  \cite{BL07} are non-gradient Sol and Nil expanding solitons, respectively.

In recent years, researchers in the field have obtained several rigidity or uniqueness results concerning gradient expanding solitons. For example, Pigola-Rimoldi-Setti \cite{PRS} proved that under certain integrability assumptions of the soliton vector field complete gradient expanding solitons must be trivial.  Catino-Mastrolia-Monticelli \cite{CMM} showed that any complete noncompact gradient expanding solitons with nonnegative sectional curvature and $L^1$-integrable scalar curvature is necessarily isometric to a quotient of the Gaussian expanding soliton. 
In addition, it was shown in \cite{Cao et al}  that a Bach-flat gradient expanding  soliton with positive Ricci curvature must be rotationally symmetric; see also related recent results \cite{Kim, CaoYu,  FLi} for expanding solitons with harmonic Weyl curvature or vanishing $D$-tensor. 
In another direction, Chodosh \cite{Chod} showed that any gradient expanding Ricci soliton with positive sectional curvature which is asymptotically conical  must be rotationally symmetric, hence one of the Bryant expanding solitons on $\RR^n$.  A similar result for gradient expanding K\"ahler-Ricci solitons was proved by Chodosh-Fong \cite{CF}, namely expanding K\"ahler-Ricci solitons which have positive holomorphic bisectional curvature and are $C^2$-asymptotic to a conical K\"ahler manifold at infinity must be the $U(n)$-invariant expanding solitons constructed in \cite{Cao97}.
We also mention that Chen-Deruelle \cite{ChenDer} 
have studied the asymptotic geometry of expanding solitons with finite asymptotic curvature ratio. 
However,  for more general gradient expanding solitons we still know very little about their geometry, except in dimension $n=2$ for which there is a complete classification \cite{BM, Ramos}. 

It is well-known that, in studying complete noncompact  Riemannian manifolds, it is very crucial to have information on volume growth rate and asymptotic curvature behavior at infinity. 
Concerning the volume growth rate, Hamilton \cite {Ha05} (see also Proposition 9.46 in \cite{CLN}) showed that any $n$-dimensional complete noncompact gradient expanding soliton $(M^n, g, f)$ with nonnegative Ricci curvature $Rc\ge 0$ must have positive {\it asymptotic volume ratio} (equivalently, the maximal volume growth in view of Bishop-Gromov). Namely, for any base point $x_0\in M^n$, 
\begin{equation}
\lim_{r\to \infty} \frac {V (x_0, r)}{r^n} >0 ,
\end{equation}
where $V(x_0, r)$ denotes the volume of the geodesic ball $B(x_0, r)$ of radius $r>0$ centered at $x_0$.  Subsequently, Carrillo-Ni \cite{CarNi} extended the above result of Hamilton

\noindent to complete noncompact gradient expanding solitons with nonnegative scalar curvature $R\ge 0$ (see Proposition 5.1 in \cite{CarNi}).

As for curvature estimates,  very recently P.-Y. Chan \cite{Chan2} proved that if $(M^4, g, f)$ is a complete noncompact gradient expanding Ricci soliton with bounded scalar curvature $|R|\le R_0$ and proper potential function $f$ (i.e., $\lim_{x\to \infty} f(x) =-\infty$), then the curvature tensor $Rm$ is bounded. Note that $Rc\ge 0$ implies $0\le R\le R_0$, for some $R_0>0$, and $f$ proper (see Lemma 2.2 and Lemma 2.3).  Hence, it follows that $4$-dimensional complete gradient expanding solitons with nonnegative Ricci curvature $Rc\geq 0$  must have bounded Riemann curvature tensor $|Rm|\le C$. 

We point out  that recent progress on curvature estimates for 4-dimensional gradient Ricci solitons has been led by the important work of Munteanu-Wang \cite{MW},  in which they proved that any complete gradient shrinking soliton with bounded scalar curvature  must have bounded Riemann curvature tensor. More significantly, they showed that the Riemann curvature tensor is controlled by the scalar curvature by $|Rm|\le C R$ so that if the scalar  curvature $R$ decays at infinity so does the curvature tensor $Rm$; see also \cite {Cao et al2} for an extension, and \cite{CaoCui, Chan1} for similar estimates in the steady soliton case. 
Their curvature estimate, together with the uniqueness result of Kotschwar-Wang \cite{KW15}, has played a crucial role in the recent advance of classifying 4-dimensional complete gradient Ricci solitons, as well as in the classification of complex 2-dimensional complete gradient K\"ahler-Ricci solitons with scalar curvature going to zero at infinity by Conlon-Deruelle-Sun \cite{CDS19}.

We also remark that one of the important facts used in the work of Munteanu and Wang \cite{MW} is a very useful result of Chow-Lu-Yang \cite{CLY} on the scalar curvature lower bound $R\ge C/ f$, for some constant $C>0$, for any $n$-dimensional non-flat complete noncompact gradient shrinking soliton. Note that by the optimal asymptotic growth estimate of Cao-Zhou \cite{CaoZhou} on the potential function $f$, this is equivalent to the scalar curvature $R$ has at most quadratic decay at infinity.

Motivated by the work of Munteanu-Wang \cite{MW}, as well as Cao-Cui \cite{CaoCui} and Chan 
\cite{Chan1, Chan2}, it is natural to ask if one could also control the Riemann curvature tensor $Rm$ of a $4$-dimensional complete gradient expanding soliton with nonnegative Ricci curvature by its scalar curvature $R$. Despite some key differences with the shrinking case, notably the lack of scalar curvature quadratic (or polynomial) decay lower bound for expanding Ricci solitons, this turns out to be possible.

\begin{theorem} Let $(M^4, g, f)$ be a $4$-dimensional complete noncompact 
gradient expanding Ricci soliton with nonnegative Ricci curvature $Rc\ge 0$. Then, there exists a constant $C>0$  such that, for any $0\leq a<1$, the following estimates hold: 
\begin{equation}
 |Rm|  \le \frac {C} {1-a} R^a \quad {\mbox{and}} \quad |\nabla Rm|\le \frac {C} {(1-a)^2} R^a \quad on \ M^4 .  
\end{equation}
Moreover,  if in addition the scalar curvature $R$ has at most polynomial decay then 
\begin{equation}
 {|Rm|} \le C R  \quad on \ M^4 . 
\end{equation}
\end{theorem}

\begin{remark} As mentioned above, P.-Y. Chan \cite{Chan2} already proved the curvature estimate $|Rm|\le C$ as a special case. 
\end{remark}

\begin{remark} In Theorem 1.1, if we only assume $Rc\geq 0$ outside some compact set $K\subset M$,  then  estimates (1.3) and (1.4) are valid on $M\setminus K$ (see Theorem 5.1). 
\end{remark}

By using Theorem 1.1, we can obtain the following useful result about asymptotic curvature behavior for $4$-dimensional gradient expanding Ricci solitons with $Rc\geq 0$. 

\begin{theorem} Let $(M^4, g, f)$ be a  $4$-dimensional complete noncompact gradient expanding Ricci soliton with nonnegative Ricci curvature $Rc\ge 0$. Assume it has finite asymptotic scalar curvature ratio 
\begin{equation}
 \limsup_{r\to \infty} R  r^2< \infty, 
\end{equation}
where $r=r(x)$ is the distance function to a fixed base point in $M$. 
Then $(M^4, g, f)$ has finite {\em asymptotic curvature ratio }
\begin{equation}
A := \limsup_{r\to \infty} |Rm| r^2< \infty . 
\end{equation}
\end{theorem} 

\begin{remark} In \cite{CLX}, joint with J. Xie, we investigate curvature estimates for higher dimensional complete expanding Ricci solitons with nonnegative Ricci curvature and finite asymptotic scalar curvature ratio. 
\end{remark}

Chen and Deruelle \cite{ChenDer} have shown the existence of a $C^{1, \alpha}$ ($0<\alpha<1$) asymptotic cone structure at infinity for $n$-dimensional complete noncompact gradient expanding solitons with finite asymptotic curvature ratio $A <\infty$. 
Note that $A$ is scaling invariant and independent of the choice of any base point. 
By combining Theorem 1.2 with the result of Chen and Deruelle (Theorem 1.2, \cite{ChenDer}), we get 

\medskip
\noindent {\bf Corollary 1.3.} {\it Let $(M^4, g, f)$ be a $4$-dimensional complete noncompact non-flat 
gradient expanding Ricci soliton with nonnegative Ricci curvature $Rc\ge 0$ and finite asymptotic scalar curvature ratio.  Then $(M^4, g, f)$ has a $C^{1, \alpha}$ asymptotic cone structure at infinity, for any $\alpha \in (0, 1)$.} 

\medskip
Finally, we can also prove certain curvature estimates for $4$-dimensional complete gradient expanding solitons with bounded, positive scalar curvature and proper potential functions.

\medskip
\noindent {\bf Theorem 1.4.} {\it Let $(M^4, g, f)$ be a $4$-dimensional complete noncompact gradient expanding Ricci soliton with bounded and positive scalar curvature  $0<R\leq R_0$. Assume $f$ is proper so that 
$ \lim_{x\to \infty} f(x)=-\infty$. 
Then, for any $\alpha \in (0, \frac 12)$, we have
$$ |Rm|  \leq C_{\alpha} R^{\alpha} \quad {\mbox{and}} \quad |\nabla Rm|\leq C_{\alpha} R^{\alpha} \quad on \ M^4 $$
for some positive constant $C_{\alpha}>0$, with $C_{\alpha}\to \infty$ as $\alpha \to 1/2$. }

\medskip
The paper is organized as follows. In Section 2, we review some basic facts and collect several useful lemmas about expanding solitons that will be used later in the paper. In Section 3, we present the proofs of the main curvature estimates as stated in Theorem 1.1 and Theorem 1.2.  The proof of Theorem 1.4 is carried out in Section 4. 
Finally,  in Section 5, we briefly describe the extensions of Theorem 1.1 and Theorem 1.2 to $4$-dimensional complete gradient expanding solitons with nonnegative Ricci curvature outside a compact set and raise some questions.

\medskip 
\noindent {\bf Acknowledgements.} We would like to thank Pak-Yeung Chan, Chih-Wei Chen, Alix Deruelle, Ovidiu Munteanu, Jiaping Wang, Junming Xie and Detang Zhou for their interests and their helpful comments on an earlier version of the paper.  We also thank the referee for carefully reading our paper and providing helpful suggestions.

\section{Preliminaries}

In this section, we shall fix notations, recall some basic facts, and collect several known results about gradient expanding Ricci solitons. 

Throughout the paper, we denote by $$Rm=\{R_{ijkl}\}, \quad Rc=\{R_{ik}\},\quad R $$ the Riemann curvature tensor, the Ricci tensor, and the scalar curvature of the metric $g=g_{ij}dx^idx^j$ in local coordinates $(x^1, \cdots, x^n)$, respectively.

\begin{lemma} {\bf (Hamilton \cite{Ha95F})} Let $(M^n, g, f)$
be a complete gradient expanding Ricci soliton satisfying Eq. (1.1).
Then
\begin{equation}
 R+\Delta f =-\frac n 2 , 
\end{equation}
\begin{equation}
\nabla_iR=2R_{ij}\nabla_jf ,  
\end{equation}
\begin{equation}
R+|\nabla f|^2=-f +C_0  
\end{equation}
for some constant $C_0$.
\end{lemma}

Furthermore, replacing $f$ by $f-C_0$, we can normalize the potential function $f$ in (2.3) so that 
\begin{equation}
R+|\nabla f|^2=-f .  
\end{equation}
In the rest of the paper, we shall always assume this normalization (2.4).

Next, we have the following well-known fact about the asymptotic behavior of the potential function of a complete non-compact gradient expanding soliton with nonnegative Ricci curvature (see, e.g., Lemma 5.5 in \cite{Cao et al} or Lemma 2.2 in \cite{ChenDer}).

\begin{lemma}  Let $(M^n, g , f)$ be a complete noncompact gradient expanding Ricci soliton satisfying (1.1) and with nonnegative Ricci curvature $Rc \ge 0$. Then, there exist some constant $c_1 >0$ such that, outside some compact subset of $M^n$, the potential function $f$ satisfies the estimates
\begin{equation}
\frac 1 4(r(x)-c_1)^2\le -f(x)\le \frac 1 4 (r(x)+2\sqrt{-f(x_0)})^2 ,  
\end{equation}
where $r(x)$ is the distance function from some fixed base point in $M^n$. In particular, $-f$ is a strictly convex exhaustion function achieving its minimum at a unique interior point $x_0$, which we shall take as the base point, and the underlying manifold $M^n$ is diffeomorphic to ${\mathbb R}^n$.

\end{lemma}

\begin{remark} In (2.5), only the lower bound on $-f$ requires the assumption of nonnegative Ricci curvature. Also, the level sets of $-f$ are diffeomorphic to ${\mathbb S}^{n-1}$. 
\end{remark}

Another useful fact is the boundedness of the scalar curvature of a gradient expanding soliton with nonnegative Ricci curvature. We recently noticed that this result was already observed in Ma-Chen \cite{MC} (see also \cite{DZ, Der17}), and its  proof essentially follows Hamilton's argument in proving Theorem 20.1 in \cite{Ha95F}. 

\begin{lemma} 
Let $(M^n, g, f)$ be a complete noncompact gradient expanding Ricci soliton with nonnegative Ricci curvature $Rc \ge 0$. Then its scalar curvature $R$ is bounded from above, i.e., 
\[R\le R_0 \quad \mbox{for some positive constant} \ R_0.\]
Moreover, $R>0$ everywhere unless  $(M^n, g, f)$ is the Gaussian expanding soliton.  
\end{lemma}

\begin{remark}
In \cite{MC}, the authors asserted that either $R>0$ or $(M^n, g)$ is Ricci flat.  However, a complete noncompact Ricci-flat gradient expanding Ricci soliton must be the Gaussian expanding soliton on ${\RR}^n$ (see, e.g., Proposition 3.1 in \cite{PW}). 
\end{remark}

Furthermore, by combining  (2.1) and (2.4), we obtain 
$$\Delta f -|\nabla f|^2=f-\frac n 2.$$ Thus, setting
\begin{equation}
F=-f+\frac n 2 , 
\end{equation}
we have 
\begin{equation}
 |\na F|^2=F-R-\frac n 2 \qquad \mbox{and}  
\end{equation}
\begin{equation}
 \Delta_f F =F \qquad ({\mbox {i.e.}},  \ \Delta_f f=f-\frac n 2) , 
\end{equation}
where $\Delta_f =:\Delta -\nabla f\cdot \nabla$ is the weighted Laplace operator.

Note that, under the assumption of Lemma 2.2, the function $F(x)$ defined in (2.6) is proportional to the square of the distance function $r^2(x)$ from above and below at large distances. In the rest of the paper, we denote by
\begin{eqnarray*}
D(t)&=& \{x\in M : F(x) \leq t \}, \quad {\mbox{and}} \\ \Sigma(t) &=&\{x\in M : F(x)=t \}.
\end{eqnarray*}

We now collect several well-known differential identities on the curvatures $R, Rc$ and $Rm$ that we shall use later. 

\begin{lemma} Let $(M^n, g, f)$
be a gradient expanding Ricci soliton satisfying Eq. (1.1). Then, we have
\begin{eqnarray*}
\Delta_{f} R &=&-R-2|Rc|^2,\\
\Delta_{f} R_{ik} &=&-R_{ik} -2R_{ijkl}R_{jl},\\
\Delta_{f} {Rm} &=&  -Rm+ Rm\ast Rm,\\
\na_lR_{ijkl} &=& \na_jR_{ik}-\na_i R_{jk}=R_{ijkl}\na_lf, 
\end{eqnarray*}
where,  on the right hand side of the third equation, $Rm\ast Rm$ denotes the sum of a finite number of terms involving quadratics in $Rm$.
\end{lemma}

Based on Lemma 2.4,  one can easily derive the following differential inequalities  (see also \cite{MW, CaoCui} for the shrinking and steady ones, respectively):

\begin{lemma} Let $(M^n, g, f)$ be a complete gradient expanding soliton satisfying Eq. (1.1). Then
\begin{eqnarray*}
\Delta_{f} |Rc|^2 & \ge & 2|\na Rc|^2-2|Rc|^2-4|Rm| |Rc|^2, \\
\Delta_{f}|Rm|^2  &\ge & 2|\na Rm|^2 - 2|Rm|^2-c|Rm|^3,\\
\Delta_{f} |Rm| &\ge & -|Rm|-c|Rm|^2.
\end{eqnarray*}
Here $c>0$ is some universal constant depending only on the dimension $n$.
\end{lemma}

\begin{remark} To derive the third  inequality, one needs to use the Kato inequality
$|\nabla |Rm||\le |\nabla Rm|$.
\end{remark}

Finally,  we have the following differential inequalities on the covariant derivative $\na Rm$ of the curvature tensor (see \cite{MW} for the shrinking case). 

\begin{lemma} Let $(M^n, g, f)$ be a complete gradient expanding Ricci soliton satisfying Eq. (1.1). Then
\begin{eqnarray*}
\Delta_{f} |\na Rm|^2 &\ge & 2|\na^2 Rm|^2  -3|\na Rm|^2-c|Rm| |\na Rm|^2 \quad \mbox{and} \\
\Delta_{f} |\na Rm| &\ge & -\frac 3 2|\na Rm|-c|Rm||\na Rm|,
\end{eqnarray*}
where $c>0$ is some universal constant depending only on the dimension $n$.
\end{lemma}

\begin{proof} First of all, by commuting covariant differentiations and Lemma 2.4, 
\begin{eqnarray*} 
\Delta_{f} (\na_q Rm )& = & \na_p\na_p (\na_q Rm)-\na_p(\na_q Rm)\cdot\na_p f \\
&=& \na_p \big(\na_q \na_p Rm +Rm\ast Rm\big)-\big(\na_q\na_p Rm+R_{pq\bullet\bullet}Rm\big)\cdot\na_p f\\
&=&   \na_q (\Delta  Rm) -\na_q(\na_p Rm \na_p f)+(\na_p Rm) (\na_q\na_p f) +Rm\ast \na_q Rm \\
&=&  \na_q (\Delta_f  Rm) -\frac 1 2 \na_q Rm +Rm\ast \na_q Rm\\
&=& -\frac 3 2 \na_q Rm +Rm\ast \na_q Rm,
\end{eqnarray*}
where we have used the 4th identity in Lemma 2.4  in the third equality and the expanding soliton equation (1.1) in the 4th equality.  Thus, 
\begin{eqnarray*} 
\Delta_{f} |\na Rm|^2 & = & 2|\na^2 Rm|^2 +2 (\na Rm) \Delta_{f} (\na Rm)\\
 &\ge & 2|\na^2 Rm|^2 -3  |\na Rm|^2 -c|Rm|  |\na Rm|^2.
\end{eqnarray*}
This proves the first inequality. The second inequality follows easily from the first one and Kato's inequality. 
\end{proof}

\section{The proof of Main Results}

First of all, we need the following key fact, valid for all $4$-dimensional gradient Ricci solitons, due to Munteanu and Wang \cite{MW} (see also Lemma 1 in \cite{Chan1}).

\begin{proposition}{\bf(Munteanu-Wang \cite{MW})} Let $(M^4, g, f)$ be a $4$-dimensional complete noncompact gradient Ricci soliton.
Then, there exists some universal constant $A_0>0$ such that, at any point where $\na f\neq 0$, 
$$ |Rm|\le A_0\left(|Rc|+ \frac {|\nabla Rc|} {|\nabla f|} \right). $$
\end{proposition}

As a consequence of Proposition 3.1, Lemma 2.2 and Lemma 2.3, one gets the following

\begin{proposition} 
Let $(M^4, g, f)$ be a $4$-dimensional complete noncompact gradient expanding Ricci soliton with nonnegative Ricci curvature.
Then, for any constant $\Lambda >0$, there exists some constant $r_0>0$ (depending on $\Lambda$) such that 
\[|\na Rc|^2 \ge \frac {\Lambda} {2A_0^2} |Rm|^2 - \Lambda |Rc|^2 \quad \mbox{outside}\  D(r_0).\] 
\end{proposition}

\begin{proof} From Lemma 2.2, we know that $|\na f|(x)\neq 0$ for $x\neq x_0$. Thus, it follows from Proposition 3.1 that 
\[ |Rm|^2\le 2A_0^2 \left(|Rc|^2+ \frac {|\nabla Rc|^2} {|\nabla f|^2} \right) \qquad \mbox{on} \ M\setminus \{x_{0}\}.\]
Hence, for $x\neq x_0$, 
\begin{equation} 
\frac {|\nabla Rc|^2} {|\nabla f|^2} \ge \frac 1 {2A_0^2} |Rm|^2- |Rc|^2 . 
\end{equation}
On the other hand, by (2.6), (2.7), Lemma 2.2 and Lemma 2.3, 
\[|\na f|^2(x)= |\na F|^2(x)=F(x)-R(x)-\frac n2 \rightarrow \infty, \ \mbox{as} \  r(x)\to \infty.\]
Therefore, for any $\Lambda >0$, there exists constant $r_0>0$ such that $|\na f|^2\ge \Lambda$ outside the compact set $D(r_0)=\{F(x)\le r_0\}$, hence 
\[ |\na Rc|^2 \ge \Lambda \frac {|\nabla Rc|^2} {|\nabla f|^2} \ge  \frac {\Lambda} {2A_0^2} |Rm|^2 - \Lambda |Rc|^2.\]
\end{proof}

Now we recall the following result due to  P. Y. Chan  \cite{Chan2}.

\begin{proposition}  Let $(M^4, g, f)$ be a $4$-dimensional complete noncompact gradient expanding Ricci soliton with nonnegative Ricci curvature $Rc\ge 0$.
Then, there exists some constant $C>0$, depending on the constant $A_0$ (in Proposition 3.1), the scalar curvature upper bound $R_0$ (in Lemma 2.3) and the geometry of $D(r_0)$, such that
$$ \sup_{x\in M} |Rm|  \le C. $$
\end{proposition}

\begin {remark}  As we pointed out in the introduction, Proposition 3.3 is a special case of a more general result due to P.-Y. Chan \cite{Chan2}. For the reader's convenience, we shall provide a direct proof here. 

\end {remark}

\begin{proof} The proof follows essentially from the same arguments as in  \cite{MW} (see also the proof of Theorem 3.1 in \cite{CaoCui}). 
First of all, by the assumption of $Rc\ge0$ and Lemma 2.3, it follows that
\begin{equation}
 0\le |Rc|\le R\le R_0  .  
\end{equation}
Thus, by (3.2) and Proposition 3.2 (taking $\Lambda=4A_0^2$), we get
\begin{equation}
|\nabla Rc|^2\geq 2|Rm|^2 -4A_0^2R_0^2   
\end{equation}
outside some compact set $D(r_0)$. 

On the other hand, by using the inequalities in Lemma 2.5, (3.2) and (3.3), 
\begin{eqnarray*}
\Delta_f(|Rm|+\lambda |Rc|^2) & \geq & 2\lambda |\nabla Rc|^2-c |Rm|^2 -\lambda C_1 |Rm| - \lambda C_1\\
& \geq & (4\lambda-c) |Rm|^2 -\lambda C_1 |Rm| - \lambda C_2
\end{eqnarray*}
outside $D(r_0)$, where $C_1=C_1(R_0)>0$ depends on the scalar curvature upper bound $R_0$ and $C_2=C_2(A_0, R_0)$ depends on the constants $A_0$ and $R_0$.
By picking 
$$4\lambda=c+2,$$ it follows that, outside $D(r_0)$,
\begin{eqnarray*}
\Delta_f(|Rm|+\lambda |Rc|^2) & \ge & 2|Rm|^2-\lambda C_1 |Rm|-\lambda C_2\\
 & \ge & (|Rm|+\lambda |Rc|^2)^2 -C'_1 (|Rm|+\lambda |Rc|^2)-C'_2. 
\end{eqnarray*}
Let $u = |Rm|+\lambda |Rc|^2$, then on $M\setminus D(r_0)$
\begin{equation}
\Delta_f u \ge u^2-C'_1 u-C'_2 . 
\end{equation}

Next, let $\varphi (t)$ be a smooth  function on $\mathbb R^{+}$  defined by
$$ {\varphi (t) =\left\{
       \begin{array}{ll}
  1, \ \ \quad  \rho\le t\le 2\rho ,\\[4mm]
    0, \ \ \quad  0\le t\le \rho/2 \ \ \mbox{or} \ \ t\ge 3\rho
       \end{array}
    \right.}$$
such that 
\begin{equation}
 t^2 \left(|\varphi'(t)|^2+|\varphi''(t)|\right)\le c 
\end{equation}
for some universal constant $c$ and $\rho\geq 2r_0$ (arbitrarily large). 

Now we take $\varphi=\varphi (F(x))$ as a cut-off function, with support in $D(3\rho)$. Then, on $D(3\rho)\setminus D(\rho/2)$, 
\begin{equation}
 |\nabla \varphi|=|\varphi'\nabla F|\le  \frac {c} {\rho} |\na F| \qquad  \mbox{and}  
\end{equation}
\begin{equation}
|\Delta_f \varphi| =|\varphi' \Delta_f F +\varphi''|\nabla F|^2|\le \frac {c} {\rho}F +  \frac {c} {\rho^2}|\na F|^2\le 2c. 
\end{equation}

Set $G=\varphi^2 u$. Then direct computations, (3.4), (3.6) and (3.7) yield that
\begin{eqnarray*} \varphi^2 \Delta_f G &=&\varphi^4 \Delta_f u +G\Delta_f (\varphi^2) +2\varphi^2\nabla (G\varphi^{-2})\cdot\nabla \varphi^2 \\
&\ge &\varphi^4 \left( u^2 -C'_1u-C'_2\right) + G \left(2\varphi\Delta_f \varphi -6|\nabla\varphi|^2\right) +2\nabla G\cdot\nabla \varphi^2  \\
&\ge & G^2  -C'_1 G -C'_2 +2\nabla G\cdot\nabla \varphi^2.
\end{eqnarray*}
Now it follows from the standard maximum principle argument that  $G\le C_3$ on $D(3\rho)$ for some constant $C_3>0$  depending on $A_0$ and $R_0$, but independent of $\rho$. Hence $u=|Rm|+\lambda |Rc|^2\le C_3$ on $D(2\rho)\setminus D(\rho)$. Since $\rho\geq 2r_0$ is arbitrarily large, we see that
$$ \sup_{x\in M}|Rm|\le \sup_{x\in M} \left(|Rm|+\lambda |Rc|^2\right)\le C,$$
for some constant $C>0$ depending on $A_0$, $R_0$ and the geometry of $D(r_0)$. 
This completes the proof of  Proposition 3.3.
\end{proof}

\begin{proposition} Let $(M^4, g, f)$ be a $4$-dimensional complete noncompact gradient expanding Ricci soliton with nonnegative  Ricci curvature $Ric\ge 0$ and positive scalar curvature $R>0$.
Then, for any $0<a<1$,  we have 
\begin{eqnarray*} 
(i) \ \Delta_{f-2a\ln R} \large(R^{-a}{|Rm|}\large ) & \ge & -C R^{-a}|Rm|+a(1-a)R^{-a}|Rm| |\na \ln R|^2, \ \mbox{and} \\
(ii)\  \Delta_{f-2a\ln R} \large (R^{-2a}{|Rc|^2}\large) & \ge & 2(1-a)R^{-2a} {|\na Rc|^2}-CR^{-2a} {|Rc|^2}.
\end{eqnarray*} 
Here, $C>0$ is a constant depending on $A_0$, $R_0$ and the geometry of $D(r_0)$.
\end{proposition}

\begin{proof} By direct computation and Lemma 2.4, we have
\begin{equation}
\Delta_{f} (R^{-a})=aR^{-a} [1+2R^{-1}|Rc|^2 +(a+1)R^{-2}|\na R|^2] .  
\end{equation}
Next, by Lemma 2.5, Proposition 3.3, identity (3.8) and direct computations, 
 \begin{eqnarray*} 
 \Delta_{f} \left(\frac {|Rm|} {R^{a}} \right) & = & R^{-a} \Delta_{f} |Rm| +|Rm|\Delta_{f} (R^{-a})+2\na (R^{-a}|Rm|R^a)\cdot\na (R^{-a})\\
& \ge & -C \frac {|Rm|}  {R^{a}}+ a \frac {|Rm|} {R^{a}}\left(1+2 \frac {|Rc|^2}{R}+(a+1)|\na \ln R|^2\right)  \\
& &  -2a\na (\frac {|Rm|} {R^{a}})\cdot \na \ln R -2a^2 \frac {|Rm|} {R^{a}} |\na \ln R|^2\\\
& \ge &  -C \frac {|Rm|}  {R^{a}} + a(1-a)\frac {|Rm|} {R^{a}} |\na \ln R|^2 -2a\na \left(\frac {|Rm|} {R^{a}}\right)\cdot \na \ln R,
\end{eqnarray*}  
where $C$ depends on $A_0$, $R_0$ and the geometry of $D(r_0)$.  This proves the first inequality. 

Similarly,  by Lemma 2.5, Proposition 3.3 and (3.8),  we have
\begin{eqnarray*} 
 \Delta_{f} \left(\frac {|Rc|^2} {R^{2a}} \right) & = & R^{-2a} \Delta_{f} (|Rc|^2) +|Rc|^2\Delta_{f} (R^{-2a})+2\na |Rc|^2\cdot\na (R^{-2a})\\
& \ge & 2 \frac {|\na Rc|^2} {R^{2a}} -C \frac {|Rc|^2} {R^{2a}} +2a\frac {|Rc|^2} {R^{2a}} \left(1+2 \frac {|Rc|^2} {R}+(2a+1)|\na \ln R|^2\right)\\
&  & -2a\na \left(\frac {|Rc|^2} {R^{2a}}\right)\cdot \na \ln R -4a^2 \frac {|Rc|^2} {R^{2a}} |\na \ln R|^2-4a\frac {|Rc|} {R^{a}}\frac {|\na Rc|} {R^{a}} |\na \ln R|\\
& \ge & 2 \frac {|\na Rc|^2} {R^{2a}} -C \frac {|Rc|^2} {R^{2a}} 
 +2a\frac {|Rc|^2} {R^{2a}}|\na \ln R|^2 -4a\frac {|Rc|} {R^{a}}\frac {|\na Rc|} {R^{a}}  |\na \ln R|
\\
& & 
-2a\na \left(\frac {|Rc|^2} {R^{2a}}\right)\cdot \na \ln R.
\end{eqnarray*}  
On the other hand, 
\[ 2a \frac {|\na Rc|^2} {R^{2a}} +2a\frac {|Rc|^2} {R^{2a}}|\na \ln R|^2 \ge 4a\frac {|Rc|} {R^{a}}\frac {|\na Rc|} {R^{a}} |\na \ln R|.  \]
Therefore, 
\[\Delta_{f} \left(\frac {|Rc|^2} {R^{2a}} \right)\ge 2(1-a)  \frac {|\na Rc|^2} {R^{2a}}  -C \frac {|Rc|^2} {R^{2a}} -2a\na \left(\frac {|Rc|^2} {R^{2a}}\right)\cdot  \na \ln R,\]
which is the desired second inequality. 
\end{proof}

Now we are ready to prove the main estimate on the curvature tensor $Rm$.

\begin{theorem} Let $(M^4, g, f)$ be a $4$-dimensional complete noncompact gradient expanding Ricci soliton with nonnegative Ricci curvature $Rc\ge 0$.
Then, there exists some constant $C>0$, depending on $A_0$, $R_0$ and the geometry of  $D(r_0)$, such that for any $0<a<1$, 
$$ |Rm|  \le \frac C {1-a} R^{a} \quad \mbox{on}\ M^4. $$
 
\end{theorem}

\begin{proof} By Lemma 2.3, it suffices to consider the case when the scalar curvature $R$ is positive.  
For any $a\in (0, 1)$, let 
\[v=\frac {|Rm|}{R^a} +\frac {|Rc|^2}{R^{2a}}.\]
Then, by Proposition 3.4 and Proposition 3.2 (with $\Lambda =4A_0^2$), 
\begin{eqnarray*} 
\Delta_{f-2a\ln R} (v) & \ge  &2(1-a)  \frac {|\na Rc|^2} {R^{2a}}  -C_1 \frac {|Rm|}{R^a}-C_2 \frac {|Rc|^2} {R^{2a}}\\
& \ge & 4(1-a) \frac {|Rm|^2} {R^{2a}} -8(1-a)A_0^2 \frac {|Rc|^2}{R^{2a}} -C_1\frac {|Rm|}{R^a}-C_2 \frac {|Rc|^2} {R^{2a}}\\
& \ge & 2(1-a) v^2 - C_1 v-C_2 \qquad  \mbox{on} \ M\setminus D(r_0), 
 \end{eqnarray*} 
where  $C_1>0$ and $C_2>0$ depend on $A_0$ and $R_0$.

Set $G=\varphi^2 v$. Then,  by taking the same cut-off function $\varphi (t)$  as in the proof of Proposition 3.3,  we get
\begin{eqnarray*} \varphi^2 \Delta_ {f-2a\ln R} G &=&\varphi^4 \Delta_{f-2a\ln R} (v) +G\Delta_{f-2a\ln R}  (\varphi^2) +2\varphi^2\nabla (G\varphi^{-2})\cdot\nabla \varphi^2 \\
&\ge & 2 (1-a) G^2  -C_1 G -C_2  +G\Delta_{f-2a\ln R}  (\varphi^2) \\
& & +2\nabla G\cdot\nabla \varphi^2 -8G|\na \varphi|^2. 
\end{eqnarray*}
Note that, by  (3.6), 
\[ |\na \varphi|^2\leq \frac {c^2}{\rho^2} |\na F|^2\leq c.\]
Moreover, 
\begin{equation}
 \Delta_{f-2a\ln R}  (\varphi) = \Delta_{f}  (\varphi) +2a \na \varphi \cdot \na \ln R 
\end{equation}
and 
\begin{equation}
 |\na \ln R| \leq 2\frac {|Rc|} {R} |\na F| \leq 2|\na F| .  
\end{equation}
Hence 
\begin{equation}
|\na \varphi \cdot \na \ln R|\le \frac {2c}{\rho}|\na F|^2 \leq 2c ,  
\end{equation}
By (3.7), (3.9)  and (3.11), we obtain 
\begin{equation}
\Delta_{f-2a\ln R}  (\varphi) \ge -6c   
\end{equation}
for some universal constant $c>0$. 

Therefore,  
\[\varphi^2 \Delta_ {f-2a\ln R} G \ge 2 (1-a) G^2  -C_1 G -C_2 +2\nabla G\cdot\nabla \varphi^2.\]
Now, by the same maximum principle argument as in the proof of Proposition 3.3, it follows that $v\le \frac {C}{1-a}$ on $M^n$ for some constant $C>0$ depending on $A_0$, $R_0$ and the geometry of $D(r_0)$. Therefore, 
\[ \frac {|Rm|}{R^a} \le \frac {|Rm|}{R^a} +\frac {|Rc|^2}{R^{2a}}\le \frac {C}{1-a} \quad \mbox{on} \ M^4.\]
\end{proof}

We can also prove a similar estimate for the covariant derivative $\na Rm$ of the curvature tensor. 

\begin{theorem} Let $(M^4, g, f)$ be a $4$-dimensional  complete noncompact gradient expanding Ricci soliton with nonnegative Ricci curvature $Rc\ge 0$.
Then, there exists some constant $C>0$, depending on $A_0$, $R_0$ and the geometry of $D(r_0)$, such that for any $0\leq a<1$, 
$$ |\na Rm|  \le \frac {C} {(1-a)^2} R^{a} \quad \mbox{on}\ M^4. $$
\end{theorem}

\begin{proof} Without loss of generality,  we can again assume the scalar curvature $R>0$. 

First of all, by Lemma 2.6 and Proposition 3.3, we have 
\[\Delta_{f} |\na Rm|\ge -\frac 3 2 |\na Rm|-c|Rm||\na Rm|\ge -C |\na Rm|.\]
Then, by (3.8) and direct computations, 
\begin{eqnarray*} 
\Delta_{f}\frac {|\na Rm|}{R^a} & = & R^{-a} \Delta_{f} |\na Rm| +|\na Rm|\Delta_{f} (R^{-a})+2\na (R^{-a}|\na Rm|R^a)\cdot\na (R^{-a})\\
& \ge & -C R^{-a} |\na Rm| + a R^{-a}|\na Rm|  [1+2R^{-1}|Rc|^2 +(a+1) |\na \ln R|^2]\\
& &  -2a \na (R^{-a}{|\na Rm|})\cdot \na \ln R -2a^2 R^{-a} {|\na Rm|} \cdot |\na \ln R|^2\\
& \ge & -C \frac {|\na Rm|} {R^{a}} + a(1-a)\frac {|\na Rm|} {R^{a}} |\na \ln R|^2 -2a \na \left(\frac {|\na Rm|} {R^{a}}\right)\cdot \na \ln R. 
\end{eqnarray*} 
Also, by Lemma 2.5 and Proposition 3.3, we have 
\[\Delta_{f} |Rm|^2\ge 2|\na Rm|^2-2|Rm|^2 -c|Rm|^3\ge 2|\na Rm|^2-C|Rm|^2.\]
Then, by (3.8) and direct computations, we obtain
\begin{eqnarray*} 
 \Delta_{f} \frac {|Rm|^2} {R^{2a}}  & = & R^{-2a} \Delta_{f} (|Rm|^2) +|Rm|^2\Delta_{f} (R^{-2a})+2\na |Rm|^2\cdot\na (R^{-2a})\\
& \ge & 2 \frac {|\na Rm|^2} {R^{2a}} -C \frac {|Rm|^2} {R^{2a}} +2a\frac {|Rm|^2} {R^{2a}} \left(1+2 \frac {|Rc|^2} {R}+(2a+1)\frac{|\na  R|^2}{R^2}\right)\\
&  & -2a\na \left(\frac {|Rm|^2} {R^{2a}}\right)\cdot \frac {\na R}{R} -4a^2 \frac {|Rm|^2} {R^{2a}} \frac {|\na R|^2}{R^2}-4a\frac {|Rm|} {R^{a}}\frac {|\na Rm|} {R^{a}} \frac {|\na R|}{R}\\
& \ge & 2 \frac {|\na Rm|^2} {R^{2a}} -C \frac {|Rm|^2} {R^{2a}} -2a\na \left(\frac {|Rm|^2} {R^{2a}}\right)\cdot \frac {\na R}{R} \\
& & +2a\frac {|Rm|^2} {R^{2a}}\frac{|\na  R|^2}{R^2} -4a\frac {|Rm|} {R^{a}}\frac {|\na Rm|} {R^{a}} \frac {|\na R|}{R}\\
& \ge &  2(1-a)  \frac {|\na Rm|^2} {R^{2a}}  -C \frac {|Rm|^2} {R^{2a}} -2a\na \left(\frac {|Rm|^2} {R^{2a}}\right)\cdot \frac {\na R}{R}.
\end{eqnarray*}  

Thus, by setting 
\[ w=\frac {|\na Rm|} {R^a} + \frac {|Rm|^2} {R^{2a}},\] we have 
 \begin{eqnarray*} 
\Delta_{f-2a\ln R} (w) & \ge  &2(1-a)  \frac {|\na Rm|^2} {R^{2a}}  -C \frac {|\na Rm|}{R^a}-C \frac {|Rm|^2} {R^{2a}}\\
& \ge & (1-a) \left( \frac {|\na Rm|} {R^{a}}+  \frac {|Rm|^2} {R^{2a}}\right)^2 - 2(1-a)  \frac {|Rm|^4} {R^{4a}}
\\
& & -C\left(\frac {|\na Rm|}{R^a} +\frac {|Rm|^2} {R^{2a}}\right)\\
& \ge & (1-a) w^2 - C_1 w-C_2(1-a)^{-3},
\end{eqnarray*} 
where in the last inequality we have used Theorem 3.1. 

Again, by setting $G=\varphi^2 w$, taking the same cut-off function $\varphi (t)$ as before,  and applying the same  maximum principle argument as in the proof of Theorem 3.1, we get the estimate $w\le {C}/{(1-a)^2}$ on $M^n$  for some constant $C>0$ depending on $A_0$, $R_0$  and the geometry of $D(r_0)$. Therefore,
\[ \frac {|\na Rm|}{R^a} \le \frac {\na |Rm|}{R^a} +\frac {|Rm|^2}{R^{2a}}\le \frac {C}{(1-a)^2} \quad \mbox{on} \ M^4.\]
This completes the proof of Theorem 3.2.

\end{proof}

To finish the proof of Theorem 1.1, it remains to show the following simple fact. 

\begin{proposition} Let $(M^4, g, f)$ be a $4$-dimensional  complete noncompact gradient expanding Ricci soliton with nonnegative Ricci curvature $Rc\ge 0$. Suppose  in addition the scalar curvature $R$ has at most polynomial decay, then 
there exists some constant $C>0$ such that 
$$ {|Rm|} \le C R.$$ 
\end{proposition}

\begin{proof} 
Suppose  the scalar curvature $R$ has at most polynomial decay, say $R\ge C/{r(x)^d}$ for some positive integer $d\ge 1$ outside some compact set. Then,  by (2.4) and Lemmas 2.2-2.3, we get 
\begin{equation} 
R \ge \frac {C} {r(x)^d}  \ge \frac {C} {|\na f|^d} .   
\end{equation}
Now it follows from Proposition 3.1, (3.13) and Theorem 3.2 (with $a=1-1/d$) that 
 \begin{eqnarray*}  
|Rm|  & \le & A_0 \left( |Rc| + \frac {|\na Rc| }{|\na f|}\right) \\
& \le & A_0 ( R +  C d R^{1-1/d} R^{1/d})\le C_1 R .
\end{eqnarray*} 
\end{proof}

Finally, we prove {\bf Theorem 1.2} stated in the introduction.

\begin{proof} Recall that, by Proposition 3.1, we have 
\begin{equation}
|Rm|\le A_0\ \left(|Rc|+ \frac {|\nabla Rc|} {|\nabla f|}\right) . 
\end{equation}

On the other hand,  it follows from $Rc\geq 0$ and the assumption of the finite asymptotic scalar curvature ratio (1.5) that 
\begin{equation}
 |Rc|\leq R \leq \frac {C_1} {r^2} , 
\end{equation}
for some constant $C_1>0$. 
Moreover, by picking $a=1/2$ in Theorem 1.1, we get 
\begin{equation}
 |\na Rc|\le c|\na Rm| \leq C R^{1/2} \leq \frac {C_2}  {r} ,  
\end{equation}
while 
\begin{equation}
 |\na f|^2 =-f -R= O(r^2)     
\end{equation}
by (2.4), Lemma 2.2, and Lemma 2.3. 

Plugging (3.15)-(3.17) into (3.14) leads to 
\[|Rm| \leq \frac C {r^2}.\]
This completes the proof of Theorem 1.2. 
 \end{proof}

\section {Curvature estimates for 4D expanders with $R>0$ and proper $f$} 

In this section, we consider $4$-dimensional complete noncompact gradient expanding Ricci solitons with bounded, positive scalar curvature and proper potential function, and prove Theorem 1.4 stated in the introduction. 

Let $(M^4, g, f)$ be a $4$-dimensional  
complete noncompact gradient expanding Ricci soliton with bounded, positive scalar curvature $0<R\le R_0$ and proper potential function $f$ so that $\lim_{r(x)\to \infty}f=-\infty$. Then it follows from the work of P.-Y. Chan \cite{Chan2} (see also Remark 1.1) that 
\begin{equation}
\sup_{x\in M} |Rm|  \le C_0 
\end{equation}
for some constant $C_0>0$.  

On the other hand, by making use of Chan's curvature bound (4.1), one can obtain the following estimate for the Ricci tensor. 

\begin{proposition} Let $(M^4, g, f)$ be a $4$-dimensional complete noncompact gradient expanding Ricci soliton with bounded and positive scalar curvature  $0<R\leq R_0$. Assume $f$ is proper so that $\lim_{r(x)\to \infty}f(x)=-\infty$. Then, there exists a positive constant $C>0$  such that
\[ |Rc|^2 \leq C R \quad on \ M^4.\]
\end{proposition}

\begin{proof} We first proceed as in the proof of Proposition 3.4 (ii) to derive another differential inequality for $ \Delta_{f} (R^{-1}|Rc|^2)$.  
For any $0<a \leq 1$, by  Lemma 2.5, (3.8), curvature bound (4.1) and direct computations, we have
\begin{eqnarray*} 
 \Delta_{f} \left (\frac {|Rc|^2} {R^a} \right) & = & R^{-a} \Delta_{f} (|Rc|^2) +|Rc|^2\Delta_{f} (R^{-a})+2\na |Rc|^2\cdot\na (R^{-a})\\
& \ge & 2 \frac {|\na Rc|^2} {R^a} -(4C_0+2) \frac {|Rc|^2} {R^a} + a  \frac {|Rc|^2} {R^a} \left( 2 \frac {|Rc|^2} R+ (a+1)|\na \ln R|^2\right)\\
&  &-4a\frac {|Rc||\na Rc|} {R^a} |\na \ln R\large|.\\
\end{eqnarray*} 
Since 
\[4a\frac {|Rc||\na Rc|} {R^a} |\na \ln R|\leq  \frac {4a}{a+1} \frac {|\na Rc|^2} {R^a} + a(a+1)\frac {|Rc|^2} {R^a}|\na \ln R|^2,\]
it follows that 
\begin{equation}
\Delta_{f} \left (\frac {|Rc|^2} {R^a} \right) \ge  \frac {2(1-a)}{1+a} \frac {|\na Rc|^2} {R^a} -(4C_0+2)  \frac {|Rc|^2} {R^a} + 2a  \frac {|Rc|^4} {R^{1+a}} . 
\end{equation}
Hence, by taking $a=1$ and setting $u=\frac {|Rc|^2} {R}$, we get 
\[ \Delta_{f} (u) \geq 2u^2-(4C_0+1)u.\]

Now, by using the same cut-off function and the same maximum principle argument as in the proof of Proposition 3.3, we obtain the estimate
\[ |Rc|^2 \leq C R \quad on \ M,\]
for some positive constant $C>0$. 
\end{proof} 

Moreover,  by considering a quantity of the form $\frac{|Rm|^{2p}}{R^b}+\frac{|Rc|^{2}}{R^a}$ suggested to us by P.-Y. Chan,  we can actually derive estimates on $Rm$ and $\na Rm$ in terms of the scalar curvature $R$.  Now we restate Theorem 1.4 here. 

\begin{theorem} Let $(M^4, g, f)$ be a $4$-dimensional complete noncompact gradient expanding Ricci soliton with bounded and positive scalar curvature  $0<R\leq R_0$. Assume $f$ is proper so that $\lim_{r(x)\to \infty}f(x)=-\infty$. Then, for any $\alpha \in (0, 1/2)$, 
$$ |Rm|  \leq C_{\alpha} R^{\alpha} \quad {\mbox{and}} \quad |\nabla Rm|\leq C_{\alpha} R^{\alpha} \quad on \ M^4 $$
for some positive constant $C_{\alpha}>0$, with $C_{\alpha}\to \infty$ as $\alpha \to 1/2$.  
\end{theorem}

\begin{proof} First of all,  for any $0<a<1$, from (4.2) we have 
\begin{equation}
\Delta_{f} \frac {|Rc|^2} {R^a}  \ge  \frac {2(1-a)}{1+a} \frac {|\na Rc|^2} {R^a} -(4C_0+2)  \frac {|Rc|^2} {R^a} . 
\end{equation}

On the other hand, by using Lemma 2.5, (3.8),  curvature bound (4.1) and direct computations, for any $p>0$ and $b>0$ we get  (whenever $|Rm|\neq 0$ in case $p < 1$) 
\begin{eqnarray*} 
 \Delta_{f}  \frac {|Rm|^{2p}} {R^{b}} & = & R^{-b} \Delta_{f} (|Rm|^{2p}) +|Rm|^{2p}\Delta_{f} (R^{-b})+2\na (|Rm|^{2p})\cdot\na (R^{-b})\\
& \ge & 2p(2p-1) \frac {|Rm|^{2p-2}} {R^{b}} |\na |Rm||^2 -Cp \frac {|Rm|^{2p}} {R^{b}} \\
&  & +b(b+1)\frac {|Rm|^{2p}} {R^{b}} |\na  \ln R|^2 -4pb\frac {|Rm|^{2p-1}} {R^{b}} |\na |Rm|| \ \!|\na \ln R|\\
& \ge & [2p(2p-1)-(b+1)^{-1}{4p^2}b] \frac {|Rm|^{2p-2}} {R^{b}} |\na |Rm||^2 -Cp \frac {|Rm|^{2p}} {R^{b}}.
\end{eqnarray*}  
Now,  for any $b\in (0,1)$, if we choose $2p=1+b \in (1, 2)$ then 
\[2p(2p-1)- \frac{4p^2b}{b+1}=0,\] 
hence
\begin{equation}
\Delta_{f} \frac {|Rm|^{1+b}} {R^{b}} \geq -C \frac {|Rm|^{1+b}} {R^{b}} .  
\end{equation}

By combining (4.3) and (4.4), and applying Proposition 3.2 (with $\Lambda=2A_0^2$) and Proposition 4.1, we obtain  
\begin{eqnarray*} 
 \Delta_{f} \left (\frac {|Rm|^{1+b}} {R^{b}} +\frac {|Rc|^2} {R^a} \right) & \ge & (1-a) \frac {|\na Rc|^2} {R^a} -C_1\left( \frac {|Rm|^{1+b}} {R^{b}} + \frac {|Rc|^2} {R^a}\right)\\
& \geq & (1-a) \frac {|Rm|^2} {R^a} -C_1\left( \frac {|Rm|^{1+b}} {R^{b}} + \frac {|Rc|^2} {R^a}\right) -C_2.
\end{eqnarray*}  
Finally, for  any $\alpha \in (0, 1/2)$, we choose $a=2\alpha \in (0, 1), \ b= \alpha/(1-\alpha)<a$ and set 
\[v=\frac {|Rm|^{1+b}} {R^{b}} +\frac {|Rc|^2} {R^a}.\]
Without loss of generality, by Proposition 4.1, we may assume $v\leq  2 R^{-b} |Rm|^{1+b} $. Then, we have 
\begin{equation}
 \Delta_{f} (v) \geq \frac {1-2\alpha} {4} v^{2(1-\alpha)}-C_1 v-C_2 .
\end{equation}
Since $1< 2(1-\alpha)$, by using the same cut-off function as in Proposition 3.3 and applying the standard maximum principle argument, we conclude from (4.5) that $v\leq C_{\alpha}$ for some positive constant $C_{\alpha}>0$, with $C_{\alpha}\to \infty$ as $\alpha \to 1/2$. 

Therefore, $|Rm|^{1+b} \leq C_{\alpha} {R^{b}}$,  or equivalently
\[ |Rm|\leq C_{\alpha} R^{\frac b {1+b}}= C_{\alpha} R^{\alpha} \quad on \ M^4.\]

Similarly, by using Lemma 2.6, (3.8), curvature bound (4.1) and direct computations,  we get  (whenever $|Rm|\neq 0$)
\begin{eqnarray*} 
\Delta_{f}\frac {|\na Rm|^{2p}}{R^b} & = & R^{-b} \Delta_{f} |\na Rm|^{2p} +|\na Rm|^{2p}\Delta_{f} (R^{-b})+2\na (|\na Rm|^{2p})\cdot\na (R^{-b})\\
& \geq &  [2p(2p-1)-(b+1)^{-1}{4p^2}b] \frac {|\na Rm|^{2p-2}} {R^{b}} |\na|\na Rm||^2 -Cp \frac {|\na Rm|^{2p}} {R^{b}}
\end{eqnarray*}  
and, for any $0<a<1$,
\[\Delta_{f} \frac {|Rm|^{2}} {R^{a}} \geq \frac {2(1-a)}{1+a} \frac{|\na Rm|^2}{R^a} -C \frac {|Rm|^{2}} {R^{a}}.\]
 Combining the above two inequalities and choosing $2p=1+b$, we obtain
\[ \Delta_{f} \left(\frac {|\na Rm|^{1+b}} {R^{b}} +\frac {|Rm|^2} {R^a} \right) \geq (1-a) \frac {|\na Rm|^2} {R^a} -C_3\left( \frac {|\na Rm|^{1+b}} {R^{b}} + \frac {|Rm|^2} {R^a}\right).\]

Now, for  any $\alpha \in (0, 1/2)$, we choose $a=2\alpha \in (0, 1), \ b= \alpha/(1-\alpha)<a$ so that 
\begin{eqnarray*} 
 \Delta_{f} \left (\frac {|\na Rm|^{1+b}} {R^{b}} +\frac {|Rm|^2} {R^a} \right) &  \geq & (1-2\alpha) \left(\frac {|\na Rm|^{1+b}} {R^b}\right)^{2(1-\alpha)} \\
& & - C_3\left( \frac {|\na Rm|^{1+b}} {R^{b}} + \frac {|Rm|^2} {R^a}\right) .
\end{eqnarray*}  

Since $1< 2(1-\alpha)$ and $|Rm|^2 \leq C_a R^a$, by using the same cut-off function as in Proposition 3.3 and applying the standard maximum principle argument, we conclude that 
\[ |\nabla Rm|\leq  C_{\alpha} R^{\alpha} \quad on \ M^4\]
for some positive constant $C_{\alpha}>0$ with $C_{\alpha}\to \infty$ as $\alpha \to 1/2$. 
\end{proof}

\section{Further Remarks}

\subsection{Extension of Theorem 1.1 and Theorem 1.2} \ In the more general case when the Ricci curvature of $(M^4, g, f)$ is only nonnegative outside some compact set,  the same argument as in the proof of Theorem 1.1 goes through without any changes.  Indeed, under the assumption of $Rc\geq 0$ outside a compact set $K\subset M$, it is not hard to check that the asymptotic quadratic growth estimate (2.5), outside $K$, for the potential function and the scalar curvature upper bound $R\le R_0$ (for some $R_0>0$) remain true. Moreover, the gradient estimate  (3.10) for the scalar curvature clearly holds on $M\setminus K$ as well.  Thus, we have 

\begin{theorem} Let $(M^4, g, f)$ be a $4$-dimensional complete noncompact 
gradient expanding Ricci soliton with positive scalar curvature $R>0$. Suppose the Ricci curvature of $(M^4, g, f)$ is nonnegative outside some compact set $K\subset M$. Then, for any $0\leq a<1$, there exists a constant $C>0$ (independent of $a$) such that 
\begin{equation}
 |Rm|  \le \frac {C} {1-a} R^a \quad {\mbox{and}} \quad |\nabla Rm|\le \frac {C} {(1-a)^2} R^a \quad on \ M^4 . 
\end{equation}
Moreover, suppose in addition the scalar curvature $R$ has at most polynomial decay at infinity, then 
\begin{equation}
 {|Rm|} \le C R  \quad on \ M^4 . 
\end{equation}
\end{theorem}

By Theorem 5.1, (3.17) and Proposition 3.1, we  also obtain the following result on asymptotic curvature behavior. 

\begin{theorem}  Let $(M^4, g, f)$ be a $4$-dimensional complete noncompact 
gradient expanding Ricci soliton with positive scalar curvature $R>0$. Suppose the Ricci curvature of $(M^4, g, f)$ is nonnegative outside some compact set $K\subset M$ and the asymptotic scalar curvature ratio is finite
\[  \limsup_{r\to \infty} R  r^2< \infty, \]
where $r=r(x)$ is the distance function to a fixed base point in $M$. 
Then $(M^4, g, f)$ has finite {\em asymptotic curvature ratio }
\[  \limsup_{r\to \infty} |Rm| r^2< \infty.\] 
\end{theorem}

\subsection {Open questions} 
As mentioned in the introduction, the scalar curvature lower bound $R\ge C/f$ proved by Chow-Lu-Yang \cite{CLY} for complete gradient shrinking solitons played a crucial role in the work of Munteanu-Wang \cite{MW}.  On the other hand, if there is any polynomial lower bound for $R$, then we can achieve the sharp curvature estimate $|Rm|\le CR$ in the expanding case by Theorem 1.1. 

\medskip
{\bf Question 1}. Suppose $(M^n, g, f)$ is a complete noncompact non-flat gradient expanding soliton with  nonnegative Ricci curvature $Rc\ge 0$. Does its scalar curvature $R$ have at most polynomial decay, i.e.,  $R\ge C/F^d$ for some $d\ge 2$, at infinity? 

\begin{remark} 
Recently, P.-Y. Chan \cite{Chan2} has proved an exponential decay scalar curvature lower bound for expanding solitons with positive scalar curvature  $R> 0$ and proper $f$: Suppose $(M^n, g, f)$, $n\ge 2$,  is any $n$-dimensional complete noncompact gradient expanding Ricci soliton with positive scalar curvature  $R> 0$ and proper potential function $F=-f+n/2$. Then
\[ R\ge CF^{1-\frac n2}e^{-F} \quad on \ M^n\]
for some constant $C>0$.  
\end{remark}

\begin{remark} 
Based on the constructions of Deruelle \cite{Der16, Der17}, very recently Chan and Zhu \cite{Chan3} have exhibited an example of $3$-dimensional complete  asymptotically conical expanding soliton with nonnegative curvature operator $Rm\ge 0$ such that 
\[ \liminf_{r\to \infty} \ F |Rm| =0 \quad  \mbox{and} \ \limsup_{r\to \infty} \ F |Rm| <\infty . \] 
\end{remark}

Next, gradient estimates for scalar curvature $R$ is quite essential whenever we involve the operator $\Delta_{f-2\ln R}$ in the maximum principle argument. 
In the shrinking case,  Munteanu-Wang \cite{MW} initially obtained the gradient estimate \[|\na \ln R|^2 \le C\ln (f+2),\] and used it to derive the sharp $Rc$ estimate $|Rc| \le CR$ and the estimate $|\na Rm|\le CR$.  Similarly, for Ricci expanders with $Rc\ge 0$, the gradient estimate in (3.10), 
\[|\na \ln R|^2 \le 4 |\na F|^2\le 4 F,\] 
is used to obtain the almost sharp estimate on $Rm$ and the estimate on 
$|\na Rm|$ in Theorem 1.1. 

\medskip
{\bf Question 2}. Suppose $(M^4, g, f)$ is a complete noncompact gradient expanding Ricci soliton with $0<R \le R_0$ and proper $f$. Does the gradient estimate 
\begin{equation}
|\na \ln R|\le C |\na F|
\end{equation}
(or more precisely, $R^{-1} { Rc(\na F, \na F) }  \le C$)
hold for some constant $C>0$ ? 

\begin{remark} If (5.3) holds, then one would be able to improve Theorem 1.4 (Theorem 4.1) to  $|Rm|^2\le CR$ and $|\na Rm|^2\le CR$.
\end{remark} 

Moreover, in the shrinking case, the estimate $|\na Rm | \le C R$ of Munteanu and Wang \cite{MW} implies the sharper gradient estimate 
\begin{equation}
|\na \ln R| \le C \quad \text{on} \ M .
\end{equation}
In the expanding case, from $|\na Rm| \le C R^{a} $ for any $0<a<1$, we also have 
\begin{equation}
|\na \ln R| \le C_aR^{a-1} \quad \text{on} \ M .
\end{equation}

\medskip
 {\bf Question 3}. Does gradient estimate (5.4) for the scalar curvature hold for 4-dimensional gradient expanding Ricci solitons with nonnegative Ricci curvature? 

\begin{remark}  If the answer to Question 3 is affirmative, then one would be able to improve the conclusions in Theorem 1.1 to  $|Rm|\le CR$ and $|\na Rm|\le CR$ without any extra assumption.
\end{remark}

\end{document}